\documentclass[10pt]{amsart}
\usepackage{amssymb,amsthm,amsmath}

\newtheorem{prop}{Proposition}
\newtheorem{thm}{Theorem}

\theoremstyle{definition}
\newtheorem{defn}{Definition}
\newtheorem{example}{Example}

\def\bP{{\mathbb P}}

\def\Z{{\mathbb Z}}
\def\Q{{\mathbb Q}}
\def\R{{\mathbb R}}
\def\C{{\mathbb C}}

\def\B{\mathcal{B}}

\def\G{\mathcal{G}}
\def\O{{\mathcal O}}
\def\E{{\mathcal E}}
\def\H{{\mathcal H}}
\def\cC{{\mathcal C}}
\def\D{{\mathcal D}}
\def\AS{{\mathfrak S}}

\def\CS{{\mathfrak C}}

\def\LG{\mathrm{LG}}

\def\Sp{\mathrm{Sp}}
\def\SL{\mathrm{SL}}
\def\HH{\mathrm{H}}
\def\CH{\mathrm{CH}}
\def\Id{\mathrm{Id}}
\def\rX{\mathrm{X}}
\def\inv{\mathrm{inv}}
\def\Inv{\mathrm{Inv}}
\def\Par{\Pi}
\def\l{{\lambda}}
\def\x{\mathrm{x}}

\def\L{{\Lambda}}
\def\X{{\mathfrak X}}

\def\a{{\alpha}}
\def\b{{\beta}}

\def\s{{\sigma}}
\def\om{{\varpi}}
\def\ome{{\omega}}
\def\Om{{\Omega}}

\def\r{\mathfrak{r}}

\DeclareMathOperator{\Spec}{Spec}

\newcommand\bull{{\scriptscriptstyle \bullet}}
\newcommand{\dis}{\displaystyle}

\newcommand{\ssm}{\smallsetminus}
\newcommand{\lra}{\longrightarrow}
\newcommand{\hra}{\hookrightarrow}
\newcommand{\ra}{\rightarrow}

\newcommand{\dbar}{\overline{\partial}}

\newcommand{\End}{\mbox{End}}

\newcommand{\Ker}{\mbox{Ker}}
\newcommand{\Tr}{\mbox{Tr}}

\newcommand{\ov}{\overline}
\newcommand{\noin}{\noindent}
\newcommand{\wt}{\widetilde}
\newcommand{\Pf}{\mbox{Pfaffian}}
\newcommand{\wh}{\widehat}

\newcommand{\longhookrightarrow}{\lhook\joinrel\relbar\joinrel\rightarrow}
\newcommand{\lhra}{\longhookrightarrow}

\begin{document}

\title[Schubert polynomials and Arakelov theory]
{Schubert polynomials and Arakelov theory of symplectic
flag varieties}
\author{Harry Tamvakis}
\date{April 26, 2010}
\thanks{The author was supported in part by National Science 
Foundation Grant DMS-0639033.}
\address{University of Maryland,
Department of Mathematics,
1301 Mathematics Building,
College Park, MD 20742, USA}
\email{harryt@math.umd.edu}

\begin{abstract}
Let $\X=\Sp_{2n}/B$ the flag variety of the symplectic group.  We
propose a theory of combinatorially explicit Schubert polynomials 
which represent the Schubert classes in the Borel presentation of 
the cohomology ring of $\X$.  We use these polynomials to describe the
arithmetic Schubert calculus on $\X$. Moreover, we give a method to
compute the natural arithmetic Chern numbers on $\X$, and show that
they are all rational numbers.
\end{abstract}

\maketitle 

\setcounter{section}{-1}

\section{Introduction}

\noindent
Let $\X=\Sp_{2n}/B$ be the flag variety for the symplectic group
$\Sp_{2n}$ of rank $n$. The cohomology (or Chow) ring of $\X$ has a
standard presentation, due to Borel \cite{Bo}, as a quotient ring
$\Z[\x_1,\ldots,\x_n]/I_n$, where the variables $\x_i$ come from the
characters of $B$ and $I_n$ is the ideal generated by the invariant
polynomials under the action of the Weyl group of $\Sp_{2n}$. On the
other hand, the cohomology $\HH^*(\X,\Z)$ is a free abelian group with
basis given by the classes of the Schubert varieties in $\X$.

The aim of a theory of {\em Schubert polynomials} is to provide an
explicit and natural set of polynomial representatives for the
Schubert classes in the above Borel presentation of the cohomology
ring. Using a construction of Bernstein-Gelfand-Gelfand \cite{BGG} and
Demazure \cite{D1, D2}, one has an algorithm for obtaining a family of
polynomials which represent the Schubert classes, by applying divided
difference operators to a polynomial $T_{w_0}$ representing the class
of a point in $\X$. For the usual $\SL_n$ flag varieties, Lascoux and
Sch\"utzenberger \cite{LS} observed that one special choice of
$T_{w_0}$ leads to polynomials that represent the Schubert classes
simultaneously for all sufficiently large $n$. These type A Schubert
polynomials are the most natural ones to use from the point of view of
combinatorics and of geometry; they describe the degeneracy loci of
maps of vector bundles (see \cite{F1, BKTY, KMS}) and are the
prototype for any proposed theory of Schubert polynomials in the other
Lie types.

For the purposes of the present paper, our interest in Schubert
polynomials is due to their utility in studying the deformations of
the cohomology ring of $\X$ which appear in quantum cohomology and in
the extension of Arakelov theory to higher dimensions due to Gillet
and Soul\'e \cite{GS1}. With this latter application in mind, we
observed in \cite{T2, T3, T4} that a suitable theory of polynomials should
provide a lifting of the Schubert calculus from the quotient ring
$\Z[\rX_n]/I_n$ to the ring $\Z[\rX_n]$ of all polynomials in
$\rX_n=(\x_1,\ldots,\x_n)$. In addition, one would like to have strong
control over which polynomials are contained in the ideal $I_n$ of
relations; this was called the {\em ideal property} in \cite{T3}.

The search for a good theory of Schubert polynomials in types
B, C, and D has received much attention in the past (see e.g.\
\cite{BH, FK, F2, F3, KT1, LP1, LP2}). The best understood theory from the
combinatorial point of view appears to be that of Billey and Haiman
\cite{BH}, whose Schubert polynomials form a $\Z$-basis for a
polynomial ring in infinitely many variables. Unfortunately, when
expressed in their most explicit form, the Billey-Haiman polynomials
are not suitable for the above mentioned applications,
because the variables used are not geometrically natural.

This problem is already apparent in the case of the Lagrangian
Grassmannian $\LG=\Sp_{2n}/P_n$, where $P_n$ denotes the maximal
parabolic subgroup associated to the right end root in the Dynkin
diagram of type $\text{C}_n$. The Billey-Haiman Schubert polynomials
for the Schubert classes which pull back to $\X$ from $\LG$ coincide
with the Schur $Q$-functions \cite{S}, but the latter are not
polynomials in the Chern roots of any homogeneous vector bundle over
$\LG$. For the geometric applications to degeneracy loci and
elsewhere, one may instead use the $\wt{Q}$-polynomials of Pragacz and
Ratajski \cite{PR}. Indeed, we showed in \cite{T4} and \cite{KT2} that
the latter objects control the arithmetic and quantum Schubert
calculus on $\LG$, respectively.

In the first part of this article, we introduce a new theory of
symplectic Schubert polynomials $\CS_w(\rX_n)$, which are to the
Billey-Haiman Schubert polynomials what the $\wt{Q}$-polynomials are
to the Schur $Q$-functions. The $\CS_w$ are defined as linear
combinations of products of $\wt{Q}$-polynomials with type A Schubert
polynomials, with coefficients given by combinatorially explicit
integers which appear in the Billey-Haiman theory. Furthermore, the
polynomials $\CS_w$ extend to a $\Z$-basis of the full polynomial ring
$\Z[\rX_n]$, which has the ideal property mentioned above. Although
they represent the Schubert classes in the Borel
presentation of $\HH^*(\X,\Z)$, the $\CS_w$ do not respect the divided
difference operator given by the sign change, and thus differ from the
previously known type C Schubert polynomials.

In the second half of the paper, we use the polynomials $\CS_w$ to
describe the arithmetic Schubert calculus on $\X$, in its natural
smooth Chevalley model over $\Spec\Z$. Arithmetic Schubert calculus is
concerned with the multiplicative structure of the Gillet-Soul\'e
arithmetic Chow ring $\wh{\CH}(\X)$, expressed in terms of certain (carefully
chosen!)  {\em arithmetic Schubert classes}. The present study is thus a
belated sequel to \cite{T2}, which examined arithmetic intersection
theory on $\SL_n$ flag varieties. As noted in the introduction to
\cite{T2}, the hermitian differential geometry required to develop the
theory for $\Sp_{2n}/B$ was available at that time; what was lacking
was the theory of Schubert polynomials which is provided here.

The arithmetic scheme $\X$ parametrizes, over any base field $k$, 
all partial flags of subspaces
\[
0 =E_0 \subset E_1 \subset \cdots \subset E_n \subset E_{2n} = E
\]
with $\dim(E_i)=i$ for each $i$ and $E_n$ isotropic with respect to
the skew diagonal symplectic form on $E$.  Let $\ov{E}$ be the trivial
vector bundle of rank $2n$ over $\X$ equipped with a trivial hermitian
metric on $E(\C)$ compatible with the symplectic form. The metric on
$E(\C)$ induces metrics on all the subbundles $E_i(\C)$, giving a {\em
hermitian filtration} of $\ov{E}$. For $1\leq i \leq n$, let
$\ov{L}_i$ denote the quotient line bundle $E_{n+1-i}/E_{n-i}$ with
the induced hermitian metric on $L_i(\C)$, and set
$\wh{x}_i=-\wh{c}_1(\ov{L}_i)$, where $\wh{c}_1(\ov{L}_i)$ is the
arithmetic first Chern class of $\ov{L}_i$.

Let $h\in\Z[\rX_n]$ be any polynomial in the ideal $I_n$. We provide
an algorithm to compute the arithmetic intersection
$h(\wh{x}_1,\ldots,\wh{x}_n)$ in $\wh{\CH}(\X)$, as the class of an
explicit $\Sp(2n)$-invariant differential form on $\X(\C)$. In
particular, we show that all arithmetic Chern numbers on $\X$
involving the $\wh{x}_i$ are rational numbers (Theorem \ref{mainthm}).
The key relations in the ring $\wh{\CH}(\X)$ required for this
calculation involve the {\em Bott-Chern forms} of hermitian fitrations
over $\X$. As in \cite{T2}, these differential forms are identified
with certain polynomials in the entries of the curvature matrices of
the homogeneous vector bundles over $\X$. Using a computation of
Griffiths and Schmid \cite{GrS}, the latter entries may be expressed
in terms of $\Sp(2n)$-invariant differential forms on
$\Sp(2n)$. Finally, our main result (Theorem \ref{SPring}) describes
the arithmetic Schubert calculus on $\X$ using the structure
constants for the product of two symplectic Schubert polynomials
$\CS_w$ in the polynomial ring $\Z[\rX_n]$.

The paper is organized as follows. In \S \ref{prelims} we provide some
combinatorial preliminaries on $\wt{Q}$-polynomials and the
Lascoux-Sch\"utzenberger and Billey-Haiman Schubert polynomials. We
introduce our theory of symplectic Schubert polynomials in \S
\ref{schubdef} and derive some of their basic properties in \S
\ref{sproperties}. Section \ref{hdg} includes the main facts from the
hermitian differential geometry of $\X(\C)$ that we require. The
Bott-Chern forms of hermitian filtrations are discussed in \S
\ref{bcf} and the curvature of the relevant homogeneous vector bundles
over $\X$ is computed in \S \ref{hvb}. The arithmetic intersection
theory of $\X$ is studied in \S \ref{ait}. Our method for computing
arithmetic intersections is explained in \S \ref{compaa}, the
arithmetic Schubert calculus is described in \S \ref{asc}, while \S
\ref{iacr} examines the invariant arithmetic Chow ring of $\X$.  The
theory developed in these sections can be used to compute the Faltings
height of $\X$ under its pluri-Pl\"ucker embedding in projective
space; this application is given in \S \ref{hgts}. Section \ref{examp}
works out the example of $\Sp_4/B$ explicitly.

The results of this article on Schubert polynomials and arithmetic
intersection theory have analogues for the orthogonal flag varieties.
This theory is discussed in detail in \cite{T5}. Part of this project
was announced at the Oberwolfach meeting on Arakelov Geometry in
September of 2005. I wish to thank the organizers Jean-Benoit Bost,
Klaus K\"unnemann, and Damian Roessler for making this stimulating
event possible.

\section{Preliminary definitions}
\label{prelims}

\subsection{$\wt{Q}$- and $Q$-functions}
\label{definitions}
A partition $\l=(\l_1,\ldots,\l_r)$ is a finite sequence of weakly
decreasing nonnegative integers; the set of all partitions is denoted
$\Par$. 
The sum $\sum \l_i$ of the parts of $\l$ is the {\em weight} $|\l|$
and the number of (nonzero) parts $\l_i$ is the {\em length}
$\ell(\l)$ of $\l$. We set $\l_r=0$ for any $r>\ell(\l)$. 
A partition is {\em strict} if all its nonzero parts are distinct.
Let $\G_n=\{\l\in \Pi \ |\ \l_1 \leq n\}$ and let $\D_n$ be the set
of strict partitions in $\G_n$. If $\l\in \D_n$, we let
$\l'$ denote the partition in $\D_n$ whose parts complement the parts
of $\l$ in the set $\{1,\ldots,n\}$.

Let $\rX=(\x_1,\x_2,\ldots)$ be a sequence of commuting independent 
variables. The elementary symmetric functions $e_k=e_k(\rX)$ are defined
by the equation
\[
\sum_{k=0}^{\infty}e_k(\rX)t^k = \prod_{i=1}^{\infty}(1+\x_it).
\]
The polynomial ring $\Lambda = \Z[e_1,e_2,\ldots]$ is the ring of 
symmetric functions in the variables $\rX$. Following Pragacz and Ratajski 
\cite{PR}, for
each partition $\l$, we define a symmetric polynomial
$\wt{Q}_{\l}\in \Lambda$ as follows: initially, set $\wt{Q}_k=e_k$ for
$k\geq 0$.  For $i,j$ nonnegative integers, let
\[
\wt{Q}_{i,j}=
\wt{Q}_i\wt{Q}_j+
2\sum_{r=1}^j(-1)^r\wt{Q}_{i+r}\wt{Q}_{j-r}.
\]
If $\l$ is a partition of length greater than two and $m$ is the least
positive integer with $2m\geq \ell(\l)$, then set
\[
\dis
\wt{Q}_{\l}=\Pf(\wt{Q}_{\l_i,\l_j})_{1\leq i<j\leq 2m}.
\]
These $\wt{Q}$-functions are modelled on Schur's $Q$-functions \cite{S},
and enjoy the following properties:

\medskip

(a) The $\wt{Q}_{\l}(\rX)$ for $\l\in\Pi$ form a $\Z$-basis of $\L$.

\medskip

(b) $\wt{Q}_{k,k}(\rX)=e_k(\rX^2) := e_k(\x_1^2,\x_2^2,\ldots)$ for all $k$.

\medskip

(c) If $\l=(\l_1,\ldots,\l_r)$ and
$\l^+=\l\cup(k,k)=(\l_1,\ldots,k,k,\ldots,\l_r)$ then
\[
\wt{Q}_{\l^+}=\wt{Q}_{k,k}\wt{Q}_{\l}.
\]

\medskip

(d) The coefficients of $\wt{Q}_{\l}(\rX)$ are nonnegative integers.

\medskip
\noin
Let $\L_n=\Z[\x_1,\ldots,\x_n]^{S_n}$ be the ring of symmetric polynomials
in $\rX_n=(\x_1,\ldots,\x_n)$. Then we have an additional property

\medskip

(e) If $\l_1>n$, then $\wt{Q}_{\l}(\rX_n)=0$. The 
$\wt{Q}_{\l}(\rX_n)$ for $\l\in\G_n$ form a $\Z$-basis of $\L_n$.

\medskip
Suppose that $Y= (y_1,y_2,\ldots)$ is a second sequence of variables
and define symmetric functions $q_k(Y)$ by using the generating series
\[
\sum_{k=0}^{\infty}q_k(Y)t^k = \prod_{i=1}^{\infty}\frac{1+y_it}{1-y_it}.
\]
Let $\Gamma = \Z[q_1,q_2\ldots]$ and define a ring homomorphism
$\eta:\L \ra \Gamma$ by setting $\eta(e_k(\rX)) = q_k(Y)$ for each
$k\geq 1$.  J\'ozefiak \cite{J} showed that the kernel of $\eta$ is
the ideal generated by the $e_k(\rX^2)$ for $k>0$; it follows that
$\eta(\wt{Q}_{\l}) = 0$ unless $\l$ is a strict partition. Moreover,
if $p_k(\rX) = \x_1^k+\x_2^k+\cdots$ denotes the $k$-th power sum,
then we have $\eta(p_k(\rX)) = 2\,p_k(Y)$, if $k$ is odd, and
$\eta(p_k(\rX))=0$, if $k>0$ is even.  For any strict partition $\l$,
the Schur $Q$-function $Q_{\l}(Y)$ may be defined as the image of
$\wt{Q}_{\l}(\rX)$ under $\eta$.  The $Q_\l$ for strict partitions
$\l$ have nonnegative integer coefficients and form a free $\Z$-basis
of $\Gamma$.

\subsection{Divided differences and type A Schubert polynomials}
\label{Wsec}
Let $S_n$ denote the symmetric group of permutations of the set
$\{1,\ldots,n\}$; the elements $\om$ of $S_n$ are written in
single-line notation as $(\om(1),\om(2),\ldots,\om(n))$ (as usual we
will write all mappings on the left of their arguments).  Now $S_n$ is the
Weyl group for the root system $\text{A}_{n-1}$ and is generated by
the simple transpositions $s_i$ for $1\leq i\leq n-1$, where $s_i$
interchanges $i$ and $i+1$ and fixes all other elements of
$\{1,\ldots,n\}$.

The hyperoctahedral group $W_n$ is the Weyl group for the root system
$\text{C}_n$. The elements of $W_n$ are permutations with a sign attached to
each entry, and we will adopt the notation where a bar is written over
an element with a negative sign.  $W_n$ is an extension of $S_n$ by an
element $s_0$ which acts on the right by
\[
(u_1,u_2,\ldots,u_n)s_0=(\ov{u}_1,u_2,\ldots,u_n).
\]
A {\em reduced word} of $w\in W_n$ is a sequence $a_1\ldots a_r$ of
elements in $\{0,1,\ldots,n-1\}$ such that $w=s_{a_1}\cdots s_{a_r}$
and $r$ is minimal (so equal to the length $\ell(w)$ of $w$).  The
elements of maximal length in $S_n$ and $W_n$ are
\[
\om_0=(n,n-1,\ldots,1) \  \ \ \mathrm{and} \ \ \ 
w_0=(\ov{1},\ov{2},\ldots,\ov{n})
\]
respectively.

The group $W_n$ acts on the ring $\Z[\rX_n]$ of polynomials in
$\rX_n$: the transposition $s_i$ interchanges $\x_i$ and $\x_{i+1}$
for $1\leq i\leq n-1$, while $s_0$ replaces $\x_1$ by $-\x_1$ (all
other variables remain fixed). The ring of invariants
$\Z[\rX_n]^{W_n}$ is the ring of polynomials in $\Z[\rX_n]$ symmetric
in $\rX_n^2=(\x_1^2,\ldots,\x_n^2)$.  Following \cite{BGG} and
\cite{D1} \cite{D2}, there are divided difference operators
$\partial_i: \Z[\rX_n]\ra \Z[\rX_n]$. For $1\leq i\leq n-1$ they are
defined by
\[
\partial_i(f)=(f-s_if)/(\x_i-\x_{i+1})
\]
while 
\[
\partial_0(f)=(f-s_0f)/(2\x_1),
\]
for any $f\in \Z[\rX_n]$. For each $w\in W_n$, define an operator
$\partial_w$ by setting
\[
\partial_w=\partial_{a_1}\circ \cdots \circ \partial_{a_r}
\]
if $a_1\ldots a_r$ is a reduced word of $w$.

For every permutation $\om\in S_n$, Lascoux and Sch\"utzenberger
\cite{LS} defined a {\em type A Schubert polynomial}
$\AS_{\om}(\rX_n)\in\Z[\rX_n]$ by
\[
\AS_{\om}(\rX_n)=\partial_{\om^{-1}\om_0}\left(\x_1^{n-1}\x_2^{n-2}\cdots
\x_{n-1} \right).
\]
This definition is stable under the natural inclusion of $S_n$ into
$S_{n+1}$, hence the polynomial $\AS_w$ makes sense for $w\in
S_{\infty}= \cup_{n=1}^\infty S_n$. The set $\{\AS_w\}$ for $w \in
S_{\infty}$ is a free $\Z$-basis of $\Z[\rX]=\Z[\x_1,\x_2,\ldots]$.
Furthermore, the coefficients of $\AS_w$ are nonnegative integers
with combinatorial significance (see e.g.\ \cite{BJS, BKTY}).

\subsection{Billey-Haiman Schubert polynomials}
We regard $W_n$ as a subgroup of $W_{n+1}$ in the obvious way and let
$W_\infty$ denote the union $\cup_{n=1}^\infty W_n$.  Let
$Z=(z_1,z_2,\ldots)$ be a third sequence of commuting variables.  In
their fundamental paper \cite{BH}, Billey and Haiman defined a family
$\{\cC_w\}_{w\in W_\infty}$ of Schubert polynomials of type C, which
are a free $\Z$-basis of the ring $\Gamma[Z]$. The
expansion coefficients for a product $\cC_u\cC_v$ in the basis of
Schubert polynomials agree with the Schubert structure constants on
symplectic flag varieties for sufficiently large $n$. For every $w\in
W_n$ there is a unique expression
\begin{equation}
\label{bheq}
\cC_w = \sum_{{\l \, \text{strict}}\atop{\om\in S_n}}e_{\l,\om}^w
Q_\l(Y)\AS_\om(Z)
\end{equation}
where the coefficients $e_{\l,\om}^w$ are nonnegative integers. 
We proceed to give a combinatorial interpretation of these 
numbers.

A sequence $a=(a_1,\ldots,a_m)$ is called {\em unimodal} if for some
$r$ with $0\leq r \leq m$, we have
\[
a_1 > a_2 > \cdots > a_r < a_{r+1} < \cdots < a_m.
\]

Let $w\in W_n$ and $\l$ be a Young diagram with $r$ rows such that
$|\l| = \ell(w)$. A {\em Kra\'skiewicz tableau} for $w$ of shape $\l$
is a filling $T$ of the boxes of $\lambda$ with nonnegative integers
in such a way that

\medskip
\noindent
a) If $t_i$ is the sequence of entries in the $i$-th row of $T$,
reading from left to right, then the row word $t_r\ldots t_1$ is
a reduced word for $w$.

\medskip
\noindent
b) For each $i$, $t_i$ is a unimodal subsequence of maximum length
in $t_r \ldots t_{i+1} t_i$.

\begin{example} 
Let $\l\in\D_n$, $\ell=\ell(\l)$, and $k=n-\ell=\ell(\l')$. The barred
permutation
\[
w_{\l}=(\ov{\l}_1,\ldots,\ov{\l}_{\ell},\l'_k,\ldots,\l'_1)
\]
is the {\em maximal Grassmannian element} of $W_n$ corresponding to
$\l$. There is a unique Kra\'skiewicz tableau for $w_{\l}$, which has
shape $\l$, and is given as in the following example, for $\l =
(7,4,3)$:
\[
\begin{array}{l}
6\ 5 \ 4 \ 3 \ 2 \ 1 \ 0 \\
3\ 2 \ 1 \ 0 \\
2\ 1 \ 0.
\end{array}
\]
\end{example}


%

%


According to \cite[Theorem 3]{BH}, the polynomial $\cC_w$ satisfies
\[
\cC_w = \sum_{uv=w} F_u(Y) \AS_v(Z),
\]
summed over all reduced factorizations $uv=w$ in $W_\infty$ (i.e.,
such that $\ell(u)+\ell(v) = \ell(w)$) with $v\in S_\infty$. The left
factors $F_u(Y)$ are type C Stanley symmetric functions of \cite{BH,
FK, La}. In addition, Lam \cite{La} has shown that for any $u\in
W_\infty$,
\[
F_u(Y) = \sum_{\l} c^u_{\l} \, Q_{\l}(Y)
\]
where $c^u_{\l}$ equals the number of Kra\'skiewicz tableaux for $u$
of shape $\l$.  By combining these two facts, we deduce the next
result.

\begin{prop}[BH, La]
\label{BHL}
For every $w \in W_\infty$, the coefficient $e^w_{\l,\om}$ in {\em
(\ref{bheq})} is equal to the number of Kra\'skiewicz tableaux for
$w\om^{-1}$ of shape $\l$, if $\ell(w\om^{-1}) = \ell(w) - \ell(\om)$,
and equal to zero otherwise.
\end{prop}

\section{Symplectic Schubert polynomials}
\label{ssp}

\subsection{Isotropic flags and Schubert varieties}
\label{spflag}
Consider the vector space $\C^{2n}$ with its canonical basis
$\{e_i\}_{i=1}^{2n}$ of unit coordinate vectors. We define the
{\em skew diagonal symplectic form} $[\ \,,\ ]$ on $\C^{2n}$ by setting
$[e_i,e_j]=0$ for $i+j\neq 2n+1$ and $[e_i,e_{2n+1-i}]=1$
for $1\leq i \leq n$. The symplectic group $\Sp_{2n}(\C)$ is 
the group of linear automorphisms of $\C^{2n}$ preserving 
the symplectic form. The upper triangular matrices in $\Sp_{2n}$
form a Borel subgroup $B$.

An $n$-dimensional subspace $\Sigma$ of $\C^{2n}$
is called {\em Lagrangian} if the restriction of the symplectic 
form to $\Sigma$ vanishes. Let $\X=\Sp_{2n}/B$ be the variety 
parametrizing flags of subspaces 
\[
0= E_0 \subset E_1 \subset \cdots \subset E_n \subset E_{2n} = \C^{2n}
\]
with $\dim E_i = i$ and $E_n$ Lagrangian. Each such flag can be
extended to a complete flag $E_{\bull}$ in $\C^{2n}$ by letting
$E_{n+i}=E_{n-i}^{\perp}$ for $1\leq i\leq n$; we will call such a
flag a {\em complete isotropic flag}.  The same notation is used to
denote the tautological flag $E_\bull$ of vector bundles over
$\X$. 

There is a group monomorphism $\phi:W_n\hra S_{2n}$ with image
\[
\phi(W_n)=\{\,\om\in S_{2n} \ |\ \om(i)+\om(2n+1-i) = 2n+1,
 \ \ \text{for all}  \ i\,\}.
\]
The map $\phi$ is determined by setting, 
for each $w=(w_1,\ldots,w_n)\in W_n$ and $1\leq i \leq n$, 
\[
\phi(w)(i)=\left\{ \begin{array}{cl}
             n+1-w_{n+1-i} & \mathrm{ if } \ w_{n+1-i} \ \mathrm{is} \ 
             \mathrm{unbarred}, \\
             n+\ov{w}_{n+1-i} & \mathrm{otherwise}.
             \end{array} \right.
\]

Let $F_{\bull}$ be a fixed complete isotropic flag.  For every $w\in
W_n$ define the {\em Schubert variety} $\X_w(F_\bull)\subset \X$ as
the locus of $E_\bull \in \X$ such that
\[
\dim(E_r\cap F_s)\geq \#\,\{\,i \leq r \ |\ \phi(w)(i)> 2n-s\,\}
\ \ \mathrm{for} \ \ 1\leq r\leq n,\, 1\leq s\leq 2n.
\]
The Schubert class $\sigma_w$ in $\HH^{2\ell(w)}(\X,\Z)$ is the 
cohomology class which is Poincar\'e dual to the homology class
determined by $\X_w(F_\bull)$.

According to Borel \cite{Bo}, the cohomology ring $\HH^*(\X,\Z)$ is 
presented as a quotient 
\begin{equation}
\label{presentation}
\HH^*(\X,\Z) \cong \Z[\x_1,\ldots,\x_n]/I_n
\end{equation}
where $I_n$ is the ideal generated by all positive degree $W_n$-invariants 
in $\Z[\rX_n]$, that is, $I_n = \langle e_i(\rX_n^2),\ 1\leq i \leq n \rangle$.
The inverse of the isomorphism (\ref{presentation}) sends the class of 
$\x_i$ to $-c_1(E_{n+1-i}/E_{n-i})$ for each $i$ with $1\leq i \leq n$.

\subsection{Symplectic Schubert classes and Schubert polynomials}
\label{schubdef}
For every $\l\in \G_n$ and $\om\in S_n$, define the polynomial 
$\CS_{\l,\om}=\CS_{\l,\om}(\rX_n)$ by
\[
\CS_{\l,\om} = \wt{Q}_{\l}(\rX_n) \AS_\om(-\rX_n) = 
(-1)^{\ell(\om)}\wt{Q}_{\l}(\rX_n) \AS_\om(\rX_n).
\]
The products $\wt{Q}_{\l}(\rX_n) \AS_\om(\rX_n)$ for $\l\in\D_n$ and 
$\om\in S_n$ form a
basis for the polynomial ring $\Z[\rX_n]$ as a $\Z[\rX_n]^{W_n}$-module,
which was introduced and studied by Lascoux and Pragacz \cite{LP1}.
We observe here that the $\CS_{\l,\om}(\rX_n)$ for $\l\in\G_n$ and 
$\om\in S_n$ form a basis of $\Z[\x_1,\ldots,\x_n]$ as an abelian group. 
Moreover, properties (b) and (c) of $\wt{Q}$-polynomials ensure that 
for $\l\in\G_n\ssm\D_n$, we have $\CS_{\l,\om}(\rX_n)\in I_n$.

\begin{defn}
\label{cdefine}
For $w\in W_n$, define the symplectic Schubert polynomial 
$\CS_w=\CS_w(\rX_n)$ by
\begin{equation}
\label{defequ}
\CS_w = \sum_{{\l\in \D_n}\atop{\om\in S_n}} e^w_{\l,\om} \CS_{\l,\om}(\rX_n)
\end{equation}
where the coefficients $e^w_{\l,\om}$ are the same 
as in (\ref{bheq}) and Proposition \ref{BHL}.
\end{defn}

\begin{thm}
\label{defthm}
The symplectic Schubert polynomial $\CS_w(\rX_n)$ 
is the unique $\Z$-linear combination of the $\CS_{\l,\om}(\rX_n)$ for 
$\l\in\D_n$ and $\om\in S_n$ which represents the Schubert class
$\sigma_w$ in the Borel presentation {\em (\ref{presentation})}.
\end{thm}
\begin{proof}
For each $w\in W_n$, the Billey-Haiman polynomial $\cC_w$ represents
the Schubert class $\sigma_w$ in the Borel presentation after a
certain change of variables. Recall that a partition is {\em odd} if
all its non-zero parts are odd integers. For each partition $\mu$, let
$p_\mu=\prod_ip_{\mu_i}$, where $p_r$ denotes the $r$-th power
sum. The $p_\mu(Y)$ for $\mu$ odd form a $\Q$-basis of
$\Gamma\otimes_\Z\Q$. We therefore have a unique expression
\begin{equation}
\label{ontheway}
\cC_w = \sum_{{\mu\, \text{odd}}\atop{\om\in S_n}}a_{\mu,\om}^w\,
p_{\mu}(Y)\AS_\om(Z)
\end{equation}
in the ring $\Gamma[Z]\otimes_\Z\Q$.

Let $p_{odd}=(p_1,p_3,p_5,\ldots)$.  Define a polynomial
$\cC_w(p_{odd}(\rX), \rX_{n-1})$ in the variables $p_k:=p_k(\rX)$ for
$k$ odd and $\x_1,\ldots,\x_{n-1}$ by substituting $p_k(Y)$ with
$p_k(\rX)/2$ and $z_i$ with $-\x_i$ in (\ref{ontheway}). It is shown
in \cite[\S 2]{BH} that the polynomial $\cC_w(\rX_n):=
\cC_w(p_{odd}(\rX_n),\rX_{n-1})$ obtained by setting $\x_i = 0$ for
all $i>n$ in $\cC_w(p_{odd}(\rX),\rX_{n-1})$ represents the Schubert
class $\sigma_w$ in the Borel presentation (\ref{presentation}). Using
J\'ozefiak's homomorphism $\eta$ from \S \ref{definitions}, we see
that
\[
\sum_{{\l \, \text{strict}}\atop{\om\in S_n}} e_{\l,\om}^w
\wt{Q}_\l(\rX)\AS_\om(-\rX_n)
\]
differs from $\cC_w(p_{odd}(\rX),\rX_{n-1})$ by an element in the
ideal of $\L[\rX_{n-1}]$ generated by the $e_i(\rX^2)$ for
$i>0$. Since the ideal $I_n$ is generated by the polynomials
$e_i(\rX_n^2)$, and $\wt{Q}_{\l}(\rX_n)=0$ whenever $\l_1>n$, it
follows that $\CS_w$ represents the Schubert class $\s_w$ in the
presentation (\ref{presentation}), as required.

We claim that the $\CS_{\l,\om}$ for $\l\in \G_n\ssm\D_n$ and $\om\in
S_n$ form a free $\Z$-basis of $I_n$. To see this, note that if $h$ is
an element of $I_n$ then $h(\rX_n)=\sum_ie_i(\rX_n^2)f_i(\rX_n)$ for
some polynomials $f_i\in\Z[\rX_n]$. Now each $f_i$ is a unique
$\Z$-linear combination of the $\CS_{\mu,\om}$ for $\mu\in\G_n$ and 
$\om\in S_n$, and properties (b) and
(c) of \S \ref{definitions} give
\[
e_i(\rX_n^2)\CS_{\mu,\om}(\rX_n) =
\wt{Q}_{i,i}(\rX_n)\CS_{\mu,\om}(\rX_n) =
\CS_{\mu\cup(i,i),\om}(\rX_n).
\]
We deduce that any $h\in I_n$ lies in the $\Z$-linear span of the
$\CS_{\l,\om}$ for $\l\in \G_n\ssm\D_n$ and $\om\in S_n$.  Since the
$\CS_{\l,\om}$ for $\l\in \G_n$ and $\om\in S_n$ are linearly
independent, this proves the claim and the uniqueness assertion in the
theorem.
\end{proof}

\medskip

We remark that the statement of Theorem \ref{defthm} may serve as an
alternative definition of the symplectic Schubert polynomials
$\CS_w(\rX_n)$.

\subsection{Properties of symplectic Schubert polynomials}
\label{sproperties}
We give below some basic properties of the polynomials $\CS_w(\rX_n)$.

\medskip

(a) The set
\[
\{\CS_w \ |\ w\in W_n\}
\cup \{\CS_{\l,\om}\ |\ \l\in \G_n\ssm\D_n,\ \om\in S_n\}
\]
is a free $\Z$-basis of the polynomial ring $\Z[\x_1,\ldots,\x_n]$. 
The $\CS_{\l,\om}$ for $\l\in \G_n\ssm\D_n$ and $\om\in S_n$ span the
ideal $I_n$ of $\Z[\x_1,\ldots,\x_n]$
generated by the $e_i(\rX_n^2)$ for $1\leq i \leq n$.

\medskip

(b) For every $u,v\in W_n$, we have an equation
\begin{equation}
\label{structeq}
\CS_u\cdot \CS_v = \sum_{w\in W_n}c_{uv}^w\,\CS_w + 
\sum_{{\l\in\G_n\ssm\D_n}\atop{\om\in S_n}}
c_{uv}^{\l\om}\,\CS_{\l,\om}
\end{equation}
in the ring $\Z[\x_1,\ldots,\x_n]$. The coefficients $c_{uv}^w$ are 
nonnegative integers, which vanish unless $\ell(w)=\ell(u)+\ell(v)$, 
and agree with the structure constants in the equation of 
Schubert classes 
\[
\s_u\cdot \s_v = \sum_{w\in W_n}c_{uv}^w\, \s_w,
\]
which holds in $\HH^*(\X,\Z)$. The coefficients $c_{uv}^{\l\om}$ are
integers, some of which may be negative. Equation (\ref{structeq}) provides
the sought for lifting of the Schubert calculus from the cohomology
ring $\HH^*(\X,\Z)$ to $\Z[\x_1,\ldots,\x_n]$ discussed in the Introduction.

\medskip
(c) Stability property: For each $m<n$ let $i=i_{m,n}:W_m \to W_n$ 
be the natural embedding using the first $m$ components. Then for any
$w\in W_m$ we have
\[
\left.
\CS_{i(w)}(\rX_n)\right|_{x_{m+1}=\cdots=x_n=0}\, =\, \CS_w(\rX_m).
\]

\medskip
(d) For a maximal Grassmannian element $w_\l\in W_n$, we have 
$\CS_w(\rX_n) = \wt{Q}_\l(\rX_n)$.

\medskip

(e) For $\om\in S_n$ and $w\in W_n$, we have 
\[
\partial_\om\CS_w = \begin{cases}
(-1)^{\ell(\om)}\,\CS_{w\om} & \text{if} \ \ell(w\om) = \ell(w) - \ell(\om), \\
0 & \text{otherwise}.
\end{cases}
\]

\medskip

(f) Let $v_0=w_0\om_0=(\ov{n},\ov{n-1},\ldots,\ov{1})$. Then for every
$\om\in S_n$, we have
\[
\CS_{v_0\om}(\rX_n) = \wt{Q}_{\rho_n}(\rX_n)\AS_{\om}(-\rX_n),
\]
where $\rho_n=(n,n-1,\ldots,1)$.  In particular, for the element
$w_0\in W_n$ of longest length we have
\[
\CS_{w_0}(\rX_n) = \wt{Q}_{\rho_n}(\rX_n)\AS_{\om_0}(-\rX_n)
= (-1)^{n(n-1)/2} \x_1^{n-1}\x_2^{n-2}\cdots \x_{n-1}\wt{Q}_{\rho_n}(\rX_n).
\]
Thus $\CS_{w_0}(\rX_n)$ agrees with the symplectic Schubert polynomial
of Lascoux, Pragacz, and Ratajski \cite[Appendix A]{LP1} indexed by
$w_0$.  However these two families of polynomials do not coincide,
because unlike the Schubert polynomials of loc.\ cit., the
$\CS_w$ do not respect the  divided difference operator 
$\partial_0$.

\medskip

(g) To any $w\in W_n$ we associate a strict partition $\mu$ whose
parts are the absolute values of the negative entries of $w$. Let
$v\in S_n$ be the permutation of maximal length such that there exists
a factorization $w=uv$ with $\ell(w)=\ell(u)+\ell(v)$ for some $u\in
W_n$. Then $\CS_{\mu,v}$ is the unique homogeneous summand
$e_{\l,\om}^w\CS_{\l,\om}$ of $\CS_w$ in equation (\ref{defequ}) with
the weight $|\l|$ as small as possible.

\medskip
(h) Lascoux and Pragacz \cite[\S 1]{LP1} define a
$\Z[\rX_n]^{W_n}$-linear scalar product
\[
\langle\ ,\ \rangle : \Z[\rX_n]\times \Z[\rX_n] \lra \Z[\rX_n]^{W_n}
\]
by
\[
\langle f,g \rangle = (-1)^{n(n-1)/2}\partial_{w_0}(fg)
\]
for any $f,g\in \Z[\rX_n]$. 
The set of symplectic Schubert polynomials $\{\CS_w(\rX_n)\}_{w\in W_n}$
and the polynomials $\{\CS_{\l,\om}(\rX_n)\}_{\l\in \D_n,\om\in S_n}$
form two bases for the polynomial ring $\Z[\rX_n]$ as a
$\Z[\rX_n^2]$-module.  If $\l,\mu\in\D_n$ and $\rho,\pi\in S_n$ are such
that $\ell(\rho)+\ell(\pi)\leq n(n-1)/2$, then we have the
orthogonality relation
\begin{equation}
\label{orth2}
\langle \CS_{\l,\rho},\CS_{\mu,\pi}\rangle = \left\{ \begin{array}{cl}
             1 &  \mathrm{ if } \ \ \mu=\l' \ \ \mathrm{and} \ \
                    \pi=\om_0\rho, \\
             0 &  \mathrm{ otherwise. }
       \end{array} \right.
\end{equation}
Furthermore, if $u,v\in W_n$ are 
such that $\ell(u)+\ell(v)\leq n^2$, then we have 
\begin{equation}
\label{orth}
\langle\CS_u,\CS_v \rangle = \left\{ \begin{array}{cl}
             1 &  \mathrm{ if } \ \ v=w_0u, \\
             0 &  \mathrm{ otherwise. }
       \end{array} \right.
\end{equation}

\medskip
\medskip

The proof of Theorem \ref{defthm} established that the $\CS_{\l,\om}$
for $\l\in\G_n\ssm\D_n$ and $\om\in S_n$ form a basis of $I_n$. The
remaining statements in properties (a) and (b) follow because the
$\CS_w$ for $w\in W_n$ form a basis of $\Z[\rX_n]/I_n$. Properties
(c), (d), (e), (g) may be derived from the corresponding properties of
the Billey-Haiman Schubert polynomials, and (f) is a consequence of
(e).

Equation (\ref{orth2}) is proved by factoring the divided difference operator 
$\partial_{w_0}=\partial_{v_0}\partial_{\om_0}$, where $v_0 = w_0\om_0$,
noting that 
\[
\langle\CS_{\l,\rho}(\rX_n),\CS_{\mu,\pi}(\rX_n)\rangle =
(-1)^{n(n-1)/2}
\partial_{v_0}(\wt{Q}_\l(\rX_n)\wt{Q}_\mu(\rX_n))
\partial_{\om_0}(\AS_\rho(-\rX_n)\AS_\pi(-\rX_n)),
\]
and using the corresponding properties of the inner products defined
by $\partial_{v_0}$ and $\partial_{\om_0}$, which are derived in
\cite[(10)]{LP1} and \cite[(5.4)]{M1}, respectively. The relation
(\ref{orth}) may be deduced from (\ref{orth2}) and property (g) above,
or by using the fact that the Schubert classes $\sigma_u=\CS_u(\rX_n)$
and $\sigma_v= \CS_v(\rX_n)$ satisfy
\[
\int_{\X}\sigma_u\cdot \sigma_v = \delta_{u,w_0v}.
\]

\begin{example}
The list of all symplectic Schubert polynomials $\CS_w$ for $w\in W_3$
is given in Table \ref{schubtable}. These polynomials are displayed
according to the eight orbits of the symmetric group $S_3$ on $W_3$.
The Schubert polynomials in each orbit are easily computed from the
unique one of highest degree by applying type A divided difference
operators, using property (e).  The reader should compare this table
with \cite[Table 2]{BH} and \cite[Appendix A]{LP1}.
\end{example}

\begin{example}
Let $n=6$ and consider the (non-maximal) Grassmannian element $w =
126\ov{5}\ov{4}3$. The symplectic Schubert polynomial for $w$
is given by
\begin{align*}
\CS_w & = \wt{Q}_{651}\,\AS_{145236}- (\wt{Q}_{65}+\wt{Q}_{641})\,\AS_{245136}
-\wt{Q}_{65}\,\AS_{146235} \\
& +(\wt{Q}_{64}+\wt{Q}_{541})\,\AS_{345126}+
\wt{Q}_{64}\,\AS_{246135}-\wt{Q}_{54}\,\AS_{346125}.
\end{align*}
The corresponding Billey-Haiman Schubert polynomial is given by 
\begin{align*}
\cC_w &= Q_{871} +(Q_{87}+Q_{861})\,\AS_{124356} +
(Q_{86}+Q_{851})\,\AS_{134256} + (Q_{86}+Q_{761})\,\AS_{125346} \\
&+(Q_{85}+Q_{841})\, \AS_{234156} +
(Q_{85}+Q_{751}+Q_{76})\,\AS_{135246} +
Q_{76}\,\AS_{126345} \\ 
& +(Q_{84}+Q_{75}+Q_{741})\,\AS_{235146}
+(Q_{75}+Q_{651})\,\AS_{145236} + Q_{75}\,\AS_{136245} \\
&+ (Q_{74}+Q_{65}+Q_{641})\,\AS_{245136} +
Q_{74}\,\AS_{236145} + Q_{65}\,\AS_{146235}\\ 
&+ (Q_{64}+Q_{541})\,\AS_{345126} 
 + Q_{64}\,\AS_{246135}+Q_{54}\,\AS_{346125}.
\end{align*}
\end{example}

{\small{
\begin{table}[p]
\caption{Symplectic Schubert polynomials for $w\in W_3$}
\centering
\begin{tabular}{|l|c|} \hline
$w$ & $\CS_w(\rX_3)=\sum e_{\l,\om}^w
\,\wt{Q}_{\l}(\rX_3)\,\AS_\om(-\rX_3)$ \\ \hline 
$123 = 1$ & $1$ \\
$213 = s_1$ & $\wt{Q}_1 - \AS_{213}$ \\ 
$132 = s_2$ & $\wt{Q}_1 - \AS_{132}$ \\ 
$231 = s_1s_2$ & $\wt{Q}_2 - \wt{Q}_1\,\AS_{132} + \AS_{231}$ \\
$312=s_2s_1$ & $\wt{Q}_2 - \wt{Q}_1\,\AS_{213} + \AS_{312}$ \\ 
$321 = s_1s_2s_1$ & $\wt{Q}_3+\wt{Q}_{21}
-\wt{Q}_2\,\AS_{213}-\wt{Q}_2\,\AS_{132}
+\wt{Q}_1\,\AS_{312}+\wt{Q}_1\,\AS_{231} - \AS_{321}$ \\

$\ov{1}23 = s_0$ & $\wt{Q}_1$ \\ 
$2\ov{1}3 = s_0s_1$ & $\wt{Q}_2 - \wt{Q}_1 \,\AS_{213}$ \\ 
$\ov{1}32 = s_0s_2$ & $2\,\wt{Q}_2 - \wt{Q}_1\,\AS_{132}$ \\ 
$23\ov{1}= s_0s_1s_2$ & $\wt{Q}_3-\wt{Q}_2\,\AS_{132}+\wt{Q}_1\,\AS_{231}$ \\
$3\ov{1}2 = s_0s_2s_1$ & $\wt{Q}_3+\wt{Q}_{21}-2\,\wt{Q}_2\,\AS_{213}
+\wt{Q}_1\,\AS_{312}$ \\
$32\ov{1} = s_0s_1s_2s_1$ & $\wt{Q}_{31} 
- \wt{Q}_3\,\AS_{213} -\wt{Q}_3\,\AS_{132}
- \wt{Q}_{21}\,\AS_{132} +\wt{Q}_2\,\AS_{312}+2\,\wt{Q}_2\,\AS_{231} 
- \wt{Q}_1\,\AS_{321}$ \\

$\ov{2}13 = s_1s_0$ & $\wt{Q}_2$ \\ 
$1\ov{2}3 = s_1s_0s_1$ & $\wt{Q}_3 - \wt{Q}_2\,\AS_{213}$ \\
$\ov{2}31 = s_1s_0s_2$ & $\wt{Q}_3+\wt{Q}_{21}-\wt{Q}_2\,\AS_{132}$ \\
$13\ov{2}=s_1s_0s_1s_2$ & $-\wt{Q}_3\,\AS_{132}+\wt{Q}_2\,\AS_{231}$ \\
$3\ov{2}1 = s_1s_0s_2s_1$ & $\wt{Q}_{31}-\wt{Q}_3\,\AS_{213}-
\wt{Q}_{21}\,\AS_{213} + \wt{Q}_2\,\AS_{312}$ \\
$31\ov{2}=s_1s_0s_1s_2s_1$ & $-\wt{Q}_{31}\,\AS_{132}+\wt{Q}_3\,\AS_{312}
+\wt{Q}_3\,\AS_{231}+\wt{Q}_{21}\,\AS_{231}-\wt{Q}_2\,\AS_{321}$ \\

$\ov{3}12  = s_2s_1s_0$ & $\wt{Q}_3$ \\
$1\ov{3}2 = s_2s_1s_0s_1$ & $-\wt{Q}_3\,\AS_{213}$ \\
$\ov{3}21 = s_2s_1s_0s_2$ & $\wt{Q}_{31}-\wt{Q}_3\,\AS_{132}$ \\
$12\ov{3}=s_2s_1s_0s_1s_2$ & $\wt{Q}_3\,\AS_{231}$ \\
$2\ov{3}1 = s_2s_1s_0s_2s_1$ & $-\wt{Q}_{31}\,\AS_{213}+\wt{Q}_3\,\AS_{312}$ \\
$21\ov{3}=s_2s_1s_0s_1s_2s_1$ & 
$\wt{Q}_{31}\,\AS_{231}-\wt{Q}_{3}\,\AS_{321}$ \\

$\ov{2}\ov{1}3 = s_0s_1s_0$ & $\wt{Q}_{21}$ \\
$\ov{1}\ov{2}3 = s_0s_1s_0s_1$ & $\wt{Q}_{31} - \wt{Q}_{21}\,\AS_{213}$ \\ 
$\ov{2}3\ov{1} = s_0s_1s_0s_2$ & $\wt{Q}_{31} - \wt{Q}_{21}\,\AS_{132}$ \\
$\ov{1}3\ov{2}=s_0s_1s_0s_1s_2$ & $-\wt{Q}_{31}\,\AS_{132}
+\wt{Q}_{21}\,\AS_{231}$ \\
$3\ov{2}\ov{1} = s_0s_1s_0s_2s_1$ & $\wt{Q}_{32}-\wt{Q}_{31}\,\AS_{213}
+\wt{Q}_{21}\,\AS_{312}$ \\
$3\ov{1}\ov{2}=s_0s_1s_0s_1s_2s_1$ & $-\wt{Q}_{32}\,\AS_{132}
+\wt{Q}_{31}\,\AS_{312}+\wt{Q}_{31}\,\AS_{231}-\wt{Q}_{21}\,\AS_{321}$ \\

$\ov{3}\ov{1}2 = s_0s_2s_1s_0$ & $\wt{Q}_{31}$ \\
$\ov{1}\ov{3}2 = s_0s_2s_1s_0s_1$ & $-\wt{Q}_{31}\,\AS_{213}$ \\
$\ov{3}2\ov{1} = s_0s_2s_1s_0s_2$ & $\wt{Q}_{32} - \wt{Q}_{31}\,\AS_{132}$ \\
$\ov{1}2\ov{3}=s_0s_2s_1s_0s_1s_2$ & $\wt{Q}_{31}\,\AS_{231}$ \\
$2\ov{3}\ov{1}=s_0s_2s_1s_0s_2s_1$ & $-\wt{Q}_{32}\,\AS_{213}
+\wt{Q}_{31}\,\AS_{312}$ \\
$2\ov{1}\ov{3}=s_0s_2s_1s_0s_1s_2s_1$ & $\wt{Q}_{32}\,\AS_{231}
-\wt{Q}_{31}\,\AS_{321}$ \\

$\ov{3}\ov{2}1 = s_1s_0s_2s_1s_0$ & $\wt{Q}_{32}$ \\
$\ov{2}\ov{3}1 = s_1s_0s_2s_1s_0s_1$ & $-\wt{Q}_{32}\,\AS_{213}$ \\
$\ov{3}1\ov{2}=s_1s_0s_2s_1s_0s_2$ & $-\wt{Q}_{32}\,\AS_{132}$ \\
$\ov{2}1\ov{3}=s_1s_0s_2s_1s_0s_1s_2$ & $\wt{Q}_{32}\,\AS_{231}$ \\
$1\ov{3}\ov{2}=s_1s_0s_2s_1s_0s_2s_1$ & $\wt{Q}_{32}\,\AS_{312}$ \\
$1\ov{2}\ov{3}=s_1s_0s_2s_1s_0s_1s_2s_1$ & $-\wt{Q}_{32}\,\AS_{321}$ \\

$\ov{3}\ov{2}\ov{1}=s_0s_1s_0s_2s_1s_0$ & $\wt{Q}_{321}$ \\
$\ov{2}\ov{3}\ov{1}=s_0s_1s_0s_2s_1s_0s_1$ & $-\wt{Q}_{321}\,\AS_{213}$ \\
$\ov{3}\ov{1}\ov{2}=s_0s_1s_0s_2s_1s_0s_2$ & $-\wt{Q}_{321}\,\AS_{132}$ \\
$\ov{2}\ov{1}\ov{3}=s_0s_1s_0s_2s_1s_0s_1s_2$ & $\wt{Q}_{321}\,\AS_{231}$ \\
$\ov{1}\ov{3}\ov{2}=s_0s_1s_0s_2s_1s_0s_2s_1$ & $\wt{Q}_{321}\,\AS_{312}$ \\
$\ov{1}\ov{2}\ov{3}=s_0s_1s_0s_2s_1s_0s_1s_2s_1$ & 
$-\wt{Q}_{321}\,\AS_{321}$ \\
\hline
\end{tabular} 
\label{schubtable}
\end{table}}}

\section{Hermitian differential geometry}
\label{hdg}

\subsection{Bott-Chern forms}
\label{bcf}

In this section $X$ denotes a complex manifold, and
$A^{p,q}(X)$ is the space of $\C$-valued smooth
differential forms of type $(p,q)$ on $X$. Let
$A(X)=\bigoplus_p A^{p,p}(X)$ and $A'(X)\subset A(X)$ be the
set of forms $\varphi$ in $A(X)$ which can be written as
$\varphi=\partial \eta + \dbar \eta'$ for some smooth forms 
$\eta$, $\eta'$. Define $\wt{A}(X)=A(X)/A'(X)$. Observe that 
the operator $dd^c:\wt{A}(X)\ra A(X)$ is well defined, as is 
the  cup product $\wedge\omega: \wt{A}(X)\ra\wt{A}(X)$ for 
any closed form $\omega$ in $A(X)$.

A {\em hermitian vector bundle} on $X$ is a pair $\ov{E}=(E,h)$
consisting of a holomorphic vector bundle $E$ over $X$ and a hermitian
metric $h$ on $E$. Let $D$ be the hermitian holomorphic connection of
$\ov{E}$, with curvature $K=D^2\in A^{1,1}(X,\End(E))$, and let $n$
denote the rank of $E$.  For each integer $k$ with $1\leq k\leq n$,
there is an associated Chern form
$c_k(\ov{E}):=\Tr(\bigwedge^k(\frac{i}{2\pi}K))\in A^{k,k}(X)$,
defined locally by identifying $\End(E)$ with $M_n(\C)$. We also have
the total Chern form $c(\ov{E}) = 1+\sum_{k=1}^n c_k(\ov{E})$. These
differential forms are $d$ and $d^c$ closed, and their classes in the
de Rham cohomology of $X$ are the usual Chern classes of $E$.

Let $\r=(1\leq r_1<r_2<\ldots <r_m=n)$
be an increasing sequence of natural numbers.  
A {\em hermitian filtration $\ov{\E}$ of
type $\r$} is a filtration
\begin{equation}
\label{hfil}
\E:\ 0=E_0 \subset E_1 \subset E_2 \subset\cdots \subset E_m=E
\end{equation}
of $E$ by subbundles $E_i$ with $\mathrm{rank}(E_i)=r_i$ 
for $1\leq i \leq m$, together with a choice of hermitian metrics
on $E$ and on each quotient bundle $Q_i=E_i/E_{i-1}$. We say that 
$\ov{\E}$ is {\em split} if, when $E_i$ is given the induced metric 
from $E$ for each $i$,
the sequences
\[
\ov{\E}_i \, :\, 0\ra \ov{E}_{i-1} \ra \ov{E}_i
\ra \ov{Q}_i \ra 0
\]
are split, for $1\leq i \leq m$. In this case we have an orthogonal
splitting $\dis \ov{E}=\bigoplus_{i=1}^m\ov{Q}_i$.

In \cite[Theorem 1]{T2} we showed that there is a 
unique way to attach to every hermitian
filtration of type $\r$ a form $\wt{c}(\ov{\E})$ in
$\wt{A}(X)$ in such a way that:

(i) $\dis dd^c\wt{c}(\ov{\E})=
c(\bigoplus_{i=1}^m \ov{Q}_i)-c(\ov{E})$,

(ii) For every map $f:Y\ra X$ of complex manifolds,
        $\wt{c}(f^*(\ov{\E}))=f^*\wt{c}(\ov{\E})$,

\medskip

(iii) If $\ov{\E}$ is {\em split}, then $\wt{c}(\ov{\E})=0$.

\medskip
\noin
The differential form $\wt{c}(\ov{\E})$ is called the {\em Bott-Chern
form} of the hermitian filtration $\E$ corresponding to the total
Chern class. If $m=2$, i.e., if the filtration $\E$ has length two,
then $\wt{c}(\ov{\E})$ coincides with the Bott-Chern class $\wt{c}(0
\ra \ov{Q}_1 \ra \ov{E} \ra \ov{Q}_2 \ra 0)$ defined in \cite{BC,
GS2}.

Suppose we are given a hermitian vector bundle $\ov{E}$ and a
filtration $\E$ of $E$ by subbundles as in (\ref{hfil}). Assume that
the subbundles $E_i$ are given metrics induced from the hermitian
metric on $E$ and that the quotient bundles $Q_i$ are then given the
metrics induced from the exact sequences $\ov{\E}_i$.  Consider a
local holomorphic orthonormal frame for $E$ such that the first $r_i$
elements generate $E_i$, and let $K(\ov{E}_i)$ and $K(\ov{Q}_i)$ be
the curvature matrices of $\ov{E}_i$ and $\ov{Q}_i$ with respect to
the chosen frame. Let $K_{E_i}=\frac{i}{2\pi}K(\ov{E}_i)$ and
$K_{Q_i}=\frac{i}{2\pi}K(\ov{Q}_i)$. Then we have the following
result.

\begin{prop}[T2]
\label{ratbcf}
The Bott-Chern form $\wt{c}(\ov{\E})$ is a polynomial in the entries
of the matrices $K_{E_i}$ and $K_{Q_i}$, $1\leq i \leq m$, with {\em
rational} coefficients. 
\end{prop}

The polynomial in Proposition \ref{ratbcf} may be computed by an 
explicit algorithm. In fact, one has the equation
\[
\wt{c}(\ov{\E}) = \sum_{i=2}^m\,\wt{c}(\ov{\E}_i)\wedge
c(\ov{Q}_{i+1})\wedge\cdots\wedge c(\ov{Q}_m)
\]
and the Bott-Chern forms $\wt{c}(\ov{\E}_i)$ can be evaluated using
the formulas in \cite[\S 3]{T2}. In particular, we have
$\wt{c}_1(\ov{\E}) = 0$ and $\dis \wt{c}_2(\ov{\E}) =
\sum_{i=2}^m\wt{c}_2(\ov{\E}_i)$, while $\wt{c}_p(\ov{\E})=0$ for $p>
\text{rank}(E)$.

Suppose that $\ov{E}$ is flat and $m=2$, so that $\ov{\E}$ is
equivalent to a short exact sequence
\[
\ov{\E} \ :\ 0 \to \ov{E}_1 \to \ov{E} \to \ov{Q}_1 \to 0
\]
of hermitian vector bundles, with metrics induced from $\ov{E}$. In this
situation, according to \cite[Prop.\ 3]{T1}, we have 
\begin{equation}
\label{ses}
\wt{c}_k(\ov{\E}) = (-1)^{k-1}\H_{k-1}p_{k-1}(\ov{Q}_1).
\end{equation}
Here $p_r(\ov{Q}_1) = \Tr((K_{Q_1})^r)$ denotes the $r$-th
power sum form of $\ov{Q}_1$, while $\H_0=0$ and $\H_r =
1+\frac{1}{2}+\cdots + \frac{1}{r}$ is a harmonic number for $r>0$.

\subsection{Curvature of homogeneous vector bundles}
\label{hvb}

To simplify the notation in this section, we will redefine the group
$\Sp_{2n}(\C)$ using the {\em standard symplectic form} $[\ \,,\ ]'$
on $\C^{2n}$ whose matrix $[e_i,e_j ]_{i,j}'$ on unit coordinate
vectors is $\left(\begin{array}{cc} 0 & \Id_n \\ -\Id_n & 0
\end{array}\right)$, where $\Id_n$ denotes the $n\times n$ identity
matrix.  Let $\X=\Sp_{2n}/B$ be the symplectic flag variety and
$E_\bull$ its tautological complete isotropic flag of vector bundles.
We equip the trivial vector bundle $E_{2n}=\C^{2n}_\X$ with the
trivial hermitian metric $h$ compatible with the symplectic form $[\,\
,\ ]'$ on $\C^{2n}$. More precisely, view the quaternion algebra
${\mathbb H}$ as $\C\oplus j\C$ and $\C^{2n}=\C^n\oplus j\C^n$ as a
right ${\mathbb H}$-module.  Then the metric $h$ may be defined by the
equation $h(v,w)=-[vj,w]'$ (see for example \cite[\S 7.2]{FH}).

The metric on $E$ induces metrics on all the subbundles $E_i$ and the
quotient line bundles $Q_i=E_i/E_{i-1}$, for $1\leq i\leq n$.  Our
goal here is to compute the $\Sp(2n)$-invariant curvature matrices of
the homogeneous vector bundles $\ov{E}_i$ and $\ov{Q}_i$ for $1\leq i
\leq n$.  Following \cite[\S 4]{GrS} and \cite[\S 5]{T2}, this may be
done by pulling back these matrices of $(1,1)$-forms from $\X$ to the
compact Lie group $\Sp(2n)$, where their entries may be expressed in
terms of the basic invariant forms on $\Sp(2n)$.

The Lie algebra of $\Sp_{2n}(\C)$ is given by 
$${\mathfrak s}{\mathfrak p}(2n,\C)=\{(A,B,C)\ |\ A,B,C \in
M_n(\C), \ B,C \mbox{ symmetric} \},$$ 
where $(A,B,C)$ denotes the matrix 
$\left(
\begin{array}{cc}
  A & B \\
  C & -A^t
\end{array}
\right)$.  Complex conjugation of the algebra ${\mathfrak s}{\mathfrak
p}(2n,\C)$ with respect to the Lie algebra of $\Sp(2n)$ is given by
the map $\tau$ with $\tau(A)=-\ov{A}^t$.  The Cartan subalgebra
$\mathfrak{h}$ consists of all matrices of the form
$\{(\mathrm{diag}(t_1,\ldots,t_n),0,0)\ |\ t_i\in \C\}$, where
$\mathrm{diag}(t_1,\ldots,t_n)$ denotes a diagonal matrix. Consider
the set of roots
\[
R = \{ \pm t_i \pm t_j\ |\ i \neq j \} \cup
\{\pm 2t_i\} \subset {\mathfrak h}^*
\]
and a system of positive roots 
\[
R^+=\{t_i-t_j\ |\ i<j\}\cup \{ t_p+t_q\ |\ p\leq q \},
\]
where the indices run from $1$ to $n$.  We use $ij$ to denote a
positive root in the first set and $pq$ for a positive root in the
second. The corresponding basis vectors are $e_{ij}=(E_{ij},0,0)$,
$e^{pq}=(0,E_{pq}+E_{qp},0)$ for $p< q$, and $e^{pp}=(0,E_{pp},0)$,
where $E_{ij}$ is the matrix with $1$ as the $ij$-th entry and zeroes
elsewhere.

Define $\ov{e}_{ij} = \tau(e_{ij})$, $\ov{e}^{pq} = \tau(e^{pq})$, and
consider the linearly independent set 
\[
\B'=\{e_{ij},\, \ov{e}_{ij},\, e^{pq},\, \ov{e}^{pq}\ |\ i<j ,\ p\leq q\}.
\]  
The adjoint representation of ${\mathfrak h}$ on ${\mathfrak
s}{\mathfrak p}(2n,\C)$ gives a root space decomposition
\[
{\mathfrak s}{\mathfrak p}(2n,\C) =
{\mathfrak h}\oplus \sum_{i<j}(\C\, e_{ij}\oplus \C \,\ov{e}_{ij})
\oplus \sum_{p\leq q}(\C \,e^{pq}\oplus \C \,\ov{e}^{pq}).
\]
Extend $\B'$ to a basis $\B$ of ${\mathfrak s}{\mathfrak p}(2n,\C)$ and
let $\B^*$ denote the dual basis of ${\mathfrak s}{\mathfrak
p}(2n,\C)^*$. Let $\ome^{ij}$, $\ov{\ome}^{ij}$, $\ome_{pq}$,
$\ov{\ome}_{pq}$ be the vectors in $\B^*$ which are dual to $e_{ij}$,
$\ov{e}_{ij}$, $e^{pq}$, $\ov{e}^{pq}$, respectively; we regard these
elements as left invariant complex one-forms on $\Sp(2n)$. If $p>q$ we
agree that $\ome_{pq} = \ome_{qp}$ and $\ov{\ome}_{pq} = \ov{\ome}_{qp}$.
Finally, define $\ome_{ij}=\gamma\ome^{ij}$,
$\ov{\ome}_{ij}=\gamma\ov{\ome}^{ij}$, $\ome^{pq}=\gamma\ome_{pq}$,
and $\ov{\ome}^{pq}=\gamma\ov{\ome}_{pq}$, where $\gamma$ is a
constant such that $\gamma^2=\frac{i}{2\pi}$, and set $\Om_{ij} =
\ome_{ij}\wedge \ov{\ome}_{ij}$ and $\Om ^{pq} =
\ome^{pq}\wedge\ov{\ome}^{pq}$.

If $\pi:\Sp(2n)\to \X$ denotes the quotient map, the pullbacks of the
aforementioned curvature matrices under $\pi$ can now be written 
explicitly, following \cite[$(4.13)_X$]{GrS} and \cite[\S 5]{T2}.
The result is recorded in the following proposition.

\begin{prop}
\label{grsprop} For every $k$ with $1\leq k \leq n$ we have
\[
c_1(\ov{Q}_k)=\sum_{i<k}\Omega_{ik}-
\sum_{j>k}\Omega_{kj}-
\sum_{p=1}^n\Omega^{pk}
\]
and $K_{E_k}=\{\Theta_{\a\b}\}_{1\leq \a,\b\leq k}$, where
\[
\dis
\Theta_{\a\b}=-\sum_{j>k}\ome_{\a j}\wedge\ov{\ome}_{\b j}
-\sum_{p=1}^n\ome^{p\a}\wedge\ov{\ome}^{p\b}.
\]
\end{prop}

\medskip

 Let $\dis\Omega=\bigwedge_{i<j} \Omega_{ij}\wedge\bigwedge_{p\leq q}
\Omega^{pq}$.  Since the class of a point in $\X$ is Poincar\'e dual
to $\dis \prod_{k=1}^n c_1(\ov{Q}_k^*)^{2n-2k+1}$ (see e.g.\
\cite[Cor.\ 5.6]{PR}) we deduce that $\dis\int_\X\Omega=
\prod_{k=1}^n\frac{1}{(2k-1)!}$.

\section{Arithmetic intersection theory on $\Sp_{2n}/B$}
\label{ait}

\subsection{Symplectic flag varieties over $\Spec\Z$}
\label{classgps}

For the rest of this paper, $\X$ will denote the Chevalley scheme over
$\Z$ for the homogeneous space $\Sp_{2n}/B$ described in \S
\ref{spflag}. The scheme $\X$ parametrizes complete isotropic flags
$E_\bull$ of a $2n$-dimensional vector space $E$ equipped with the
skew diagonal symplectic form, over any base field. The arithmetic
symplectic flag variety $\X$ is smooth over $\Spec \Z$, and has a
decomposition into Schubert cells induced by the Bruhat decomposition
of $\Sp_{2n}$ (see e.g.\ \cite[\S 13.3]{Ja} for details).

There is a tautological complete isotropic flag of vector bundles
\[
E_\bull :\ 0=E_0\subset E_1\subset\cdots\subset E_{2n}=E
\]
over $\X$. For each $i$ with $1\leq i\leq 2n$ we let $\E_i$ denote the
short exact sequence
\[
\E_i\ :\ 0 \to E_{i-1}\to E_i \to Q_i\to 0.
\]
If $\CH(\X)$ is the Chow ring of algebraic cycles on $\X$ modulo
rational equivalence, then the class map induces an isomorphism
$\CH(\X)\cong \HH^*(\X(\C),\Z)$, following \cite[Ex.\ 19.1.11]{F4} and
\cite[Lem.\ 6]{KM}.  We deduce that there is a ring presentation
$\CH(\X)\cong \Z[\rX_n]/I_n$. The relations $e_i(\rX_n^2)$ in $I_n$
come from the Whitney sum formula applied to the filtration
$E_\bull$. This gives a Chern class equation
$\prod_{i=1}^{2n}(1+c_1(Q_i))=c(E)$ in $\CH(\X)$, which maps to the
identity $\prod_{i=1}^{2n}(1-\x_i^2) = 1$, since $E$ is a trivial
bundle.

We have an isomorphism of abelian groups 
\[
\CH(\X) \cong \bigoplus_{w\in W_n}\Z\,\CS_w(\rX_n)
\]
where the polynomial $\CS_w(\rX_n)$ represents the class of the
codimension $\ell(w)$ Schubert scheme $\X_w$ in $\X$. The latter
is defined as the closure of the corresponding Schubert cell, so 
that $\X_w(\C)$ is given as in \S \ref{spflag}.

\subsection{The arithmetic Chow group}
\label{acg}

For $p\geq 0$ we let $\wh{\CH}^p(\X)$ denote the $p$-th arithmetic
Chow group of $\X$, in the sense of Gillet and Soul\'e \cite{GS1}. 
The elements in $\wh{\CH}^p(\X)$ are represented by
arithmetic cycles $(Z,g_Z)$, where $Z$ is a codimension $p$ cycle
on $\X$ and $g_Z$ is a current of type $(p-1,p-1)$ such that the current
$dd^cg_Z+\delta_{Z(\C)}$ is represented by a smooth differential form on
$\X(\C)$. 

We let $F_{\infty}$ be the involution of $\X(\C)$ induced by complex 
conjugation. Let $A^{p,p}(\X_{\R})$ be the subspace of $A^{p,p}(\X(\C))$ 
generated by real forms $\eta$ such that $F^*_{\infty}\eta=(-1)^p\eta$;
denote by $\wt{A}^{p,p}(\X_{\R})$
the image of $A^{p,p}(\X_{\R})$ in $\wt{A}^{p,p}(\X(\C))$.
Let $A(\X_{\R})=\bigoplus_p A^{p,p}(\X_{\R})$ and
$\wt{A}(\X_{\R})=\bigoplus_p \wt{A}^{p,p}(\X_{\R})$.

Since the homogeneous space $\X$ admits a cellular decomposition, it
follows as in e.g.\ \cite{KM} that for each $p$, there is an exact
sequence
\begin{equation}
\label{sesprelim}
0 \longrightarrow  \wt{A}^{p-1,p-1}(\X_{\R})
\stackrel{a}\longrightarrow \wh{\CH}^p(\X)
\stackrel{\zeta}\longrightarrow \CH^p(\X)\longrightarrow 0,
\end{equation}
where the maps $a$ and $\zeta$ are defined by
\[
\dis
a(\eta)=(0,\eta) \qquad \textrm{ and } \qquad
\zeta(Z,g_Z)=Z.
\]
Summing (\ref{sesprelim}) over all $p$ gives the sequence
\begin{equation}
\label{seseq}
0 \longrightarrow  \wt{A}(\X_{\R})
\stackrel{a}\longrightarrow \wh{\CH}(\X)
\stackrel{\zeta}\longrightarrow \CH(\X)\longrightarrow 0.
\end{equation}

We equip $E(\C)$ with the trivial hermitian metric compatible with the
skew diagonal symplectic form $[\,\ ,\ ]$ on $\C^{2n}$.  This metric
induces metrics on (the complex points of) all the vector bundles
$E_i$ and the line bundles $L_i=E_{n+1-i}/E_{n-i}$, for $1\leq i\leq
n$. We thus obtain hermitian vector bundles $\ov{E}_i$ and line
bundles $\ov{L}_i$, together with their arithmetic Chern classes
$\wh{c}_k(\ov{E}_i)\in \wh{\CH}^k(\X)$ and $\wh{c}_1(\ov{L}_i)\in
\wh{\CH}^1(\X)$, according to \cite{GS2}. Set $\wh{x}_i =
-\wh{c}_1(\ov{L}_i)$ and for any $w\in W_n$, define
\[
\wh{\CS}_w := \CS_w(\wh{x}_1,\ldots,\wh{x}_n)\in \wh{\CH}^{\ell(w)}(\X).
\]
The unique map of abelian groups
\begin{equation}
\label{splittingmap}
\epsilon\,:\,\CH(\X)\ra \wh{\CH}(\X)
\end{equation}
sending the Schubert class $\CS_w(\rX_n)$ to $\wh{\CS}_w$ for all
$w\in W_n$ is then a splitting of (\ref{seseq}). We thus obtain an
isomorphism of abelian groups
\begin{equation}
\label{bigiso}
\wh{\CH}(\X)\cong \CH(\X)\oplus \wt{A}(\X_{\R}).
\end{equation}

\subsection{Computing arithmetic intersections}
\label{compaa}
We now describe an effective procedure for computing arithmetic Chern
numbers on the symplectic flag variety $\X$, parallel to \cite[\S
7]{T2}.  Let $c_k(\ov{E}_i)$ and $c_1(\ov{L}_i)$ be the Chern forms of
$\ov{E_i(\C)}$ and $\ov{L_i(\C)}$, respectively. In the sequel we will
identify these with their images in $\wh{\CH}(\X)$ under the inclusion
$a$.  Let $x_i = -c_1(\ov{L}_i)$ for $1\leq i \leq n$.

We begin with the short exact sequence
\[
\ov{\E}_{\LG}\ :\ 0 \to \ov{E}_n \to \ov{E} \to \ov{E}_n^* \to 0
\]
where $E_n$ denotes the tautological Lagrangian subbundle of $E$ over
$\X$. By \cite[Theorem 4.8(ii)]{GS2}, we have an equation
\begin{equation}
\label{key0}
\wh{c}(\ov{E}_n)\, \wh{c}(\ov{E}^*_n) = 1 + \wt{c}(\ov{\E}_{\LG})
\end{equation}
in $\wh{\CH}(\X)$. Consider the hermitian filtration
\[
\ov{\E} \ : \ 0 = \ov{E}_0 \subset \ov{E}_1 \subset \cdots
\subset \ov{E}_n.
\]
According to \cite[Theorem 2]{T2}, we have an equation
\begin{equation}
\label{key1}
\prod_{i=1}^n(1-\wh{x}_i) = \wh{c}(\ov{E}_n) + \wt{c}(\ov{\E}).
\end{equation}
If $\wt{c}(\ov{\E}) = \sum_i\alpha_i$ with 
$\alpha_i\in \wt{A}^{i,i}(\X_\R)$ for each $i$, then define
$\wt{c}(\ov{\E}^*) = \sum_i(-1)^{i+1}\alpha_i$. We obtain the dual equation
\begin{equation}
\label{key2}
\prod_{i=1}^n(1+\wh{x}_i) = \wh{c}(\ov{E}^*_n) + \wt{c}(\ov{\E}^*).
\end{equation}

The abelian group  $\wt{A}(\X_{\R})=\Ker\zeta$ is an ideal of $\wh{\CH}(\X)$
such that for any hermitian vector bundle $\ov{F}$ over $\X$ and 
$\eta,\eta'\in \wt{A}(\X_{\R})$, we have 
\begin{equation}
\label{modstr}
\wh{c}_k(\ov{F}) \cdot \eta = c_k(\ov{F}) \wedge \eta 
\qquad \text{and} \qquad
\eta\cdot \eta' = (dd^c\eta)\wedge\eta'.
\end{equation}
We now multiply (\ref{key1}) with (\ref{key2}) and combine the result with
(\ref{key0}) to obtain
\begin{equation}
\label{key}
\prod_{i=1}^n(1-\wh{x}^2_i) = 1 + \wt{c}(\ov{\E}, \ov{\E}^*), 
\end{equation}
where 
\begin{equation}
\label{keypart}
\wt{c}(\ov{\E}, \ov{\E}^*) = \wt{c}(\ov{\E}_{\LG}) + \wt{c}(\ov{\E})\wedge
c(\ov{E}^*_n) + \wt{c}(\ov{\E}^*)\wedge c(\ov{E}_n) +
(dd^c\wt{c}(\ov{\E}))\wedge \wt{c}(\ov{\E}^*).
\end{equation}

Using (\ref{ses}) and Proposition \ref{ratbcf}, we can express the
differential form $\wt{c}(\ov{\E}, \ov{\E}^*)$ as a polynomial in the
entries of the matrices $K_{E_i}$ and $K_{L_i}$ with rational
coefficients.  On the other hand, Proposition \ref{grsprop} gives
explicit formulas for all these curvature matrices in terms of
$\Sp(2n)$-invariant differential forms on $\X(\C)$. Note that since we
are using the skew diagonal symplectic form to define the Lie groups
in this section, the formulas in \S \ref{hvb} have to be modified
accordingly. For the matrix realization of the Lie algebra ${\mathfrak
s}{\mathfrak p}(2n,\C)$ in this case, one may consult e.g. \cite[\S
1.2, \S 2.3]{GW}, while the basis elements of $\mathfrak{h}$ should be
ordered as in \cite[(2.20)]{BH}. The indices $(i,j)$ and $(p,q)$ in
Proposition \ref{grsprop} are then replaced by $(n+1-j,n+1-i)$ and
$(n+1-q,n+1-p)$, respectively. Recalling that $L_i=E_{n+1-i}/E_{n-i}$,
we have the identities
\begin{align*}
x_1 &= -\Om_{12} - \Om_{13}- \cdots - \Om_{1n} + \Om^{11} + \Om^{12} +
\cdots + \Om^{1n} \\ 
x_2 &= \hspace{0.27cm} \Om_{12} - \Om_{23} - \cdots - \Om_{2n} +
\Om^{12} + \Om^{22} + \cdots + \Om^{2n} \\ 
&\quad \qquad \qquad \ \ \vdots
\quad \qquad \qquad \vdots \quad \qquad \qquad \vdots \\ 
x_n &= \hspace{0.27cm} \Om_{1n} +\Om_{2n} + \cdots + \Om_{n-1,n} 
+ \Om^{1n} + \Om^{2n} + \cdots + \Om^{nn}
\end{align*}
in $A^{1,1}(\X_\R)$. We also deduce the next result.
\begin{prop}
\label{compprop}
We have $\wt{c}_1(\ov{\E})=\wt{c}_1(\ov{\E}, \ov{\E}^*)=0$, \
$\dis\wt{c}_2(\ov{\E})=-\sum_{i<j}\Om_{ij}$, \ and 
\[
\wt{c}_2(\ov{\E}, \ov{\E}^*) =   
-2\,\sum_{i<j}\Om_{ij}-2\,\sum_{p<q} \Om^{pq}
-\sum_{p=1}^n\Om^{pp}.
\]
\end{prop}
\begin{proof}
We have
$\dis\wt{c}_2(\ov{\E}) = \sum_{i=2}^n\wt{c}_2(\ov{\E}_i)$, where 
$\ov{\E}_i$ is the short exact sequence
\[
\ov{\E}_i \ :\ 0 \to \ov{E}_{i-1} \to \ov{E}_i \to \ov{L}_{n+1-i} \to 0.
\]
For each $i$, write 
\[
K_{E_i}=
\left(
\begin{array}{c|c}
K^i_{11} & K^i_{12} \\ \hline
K^i_{21} & K^i_{22}
\end{array} \right)
\]
where $K^i_{11}$ is an $(i-1)\times (i-1)$ submatrix. According to 
\cite[Cor.\ 1]{T1}, we then have $\wt{c}_2(\ov{\E}_i) = c_1(\ov{E}_{i-1}) -
\Tr \, K^i_{11}$. Therefore
\[
\wt{c}_2(\ov{\E}) = \sum_{i=2}^n \left(c_1(\ov{E}_{i-1}) - \Tr \,
K^i_{11}\right) = c_1(\ov{E}_1) -\Tr\, K^n_{11} + \sum_{i=2}^{n-1}\Tr
\, K^i_{22}.
\]
Using Proposition \ref{grsprop}, we obtain 
\[
c_1(\ov{E}_1) = -\sum_{j<n} \Om_{jn} - \sum_{p=1}^n \Om^{pn},
\]
\[
-\Tr\, K^n_{11} = \sum_{p=1}^n\sum_{q=1}^{n-1}\Om^{p,n+1-q}
\]
and
\[
\Tr \, K^i_{22} = -\sum_{j>i} \Om_{n+1-j,n+1-i} - 
\sum_{p=1}^n \Om^{p,n+1-i}
\]
for $2\leq i \leq n-1$. The claimed computation of $\wt{c}_2(\ov{\E})$
follows by adding these equations.  
We deduce from
(\ref{ses}) that $\wt{c}_2(\ov{\E}_{\LG}) = -\H_1p_1(\ov{E}_n^*) =
c_1(\ov{E}_n)$, while clearly
$\wt{c}_1(\ov{\E}^*)=\wt{c}_1(\ov{\E})=0$ and
$\wt{c}_2(\ov{\E}^*)=\wt{c}_2(\ov{\E})$. 
Therefore, equation (\ref{keypart}) gives
\[
\wt{c}_2(\ov{\E}, \ov{\E}^*) =
\wt{c}_2(\ov{\E}_{\LG})+2\,\wt{c}_2(\ov{\E}) =
c_1(\ov{E}_n)+2\,\wt{c}_2(\ov{\E}).
\]
Finally, $c_1(\ov{E}_n)= \Tr\, K_{E_n}$, and the latter is computed  
using Proposition \ref{grsprop} again.
\end{proof}

\medskip
Let $h(\rX_n)$ be a homogeneous polynomial in the ideal $I_n$ of \S
\ref{spflag}. We give an effective algorithm to compute the
arithmetic intersection $h(\wh{x}_1,\ldots,\wh{x}_n)$ as a class in
$\wt{A}(\X_\R)$. First, we decompose $h$ as a sum $h(\rX_n)=\sum_i
e_i(\rX_n^2)f_i(\rX_n)$ for some polynomials $f_i$.  Equation
(\ref{key}) implies that
\begin{equation}
\label{simple}
e_i(\wh{x}_1^2,\ldots,\wh{x}_n^2) = (-1)^i \,\wt{c}_{2i}(\ov{\E}, \ov{\E}^*)
\end{equation}
for $1\leq i\leq n$. Using this and (\ref{modstr}) we see that
\begin{equation}
\label{fineq}
h(\wh{x}_1,\wh{x}_2,\ldots\wh{x}_n)=\sum_{i=1}^n
(-1)^i\,\wt{c}_{2i}(\ov{\E},\ov{\E}^*)\wedge f_i(x_1,x_2,\ldots,x_n)
\end{equation}
in $\wh{\CH}(\X)$. Now, thanks to the previous analysis, we can write
the right hand side of (\ref{fineq}) as a polynomial in the $x_i$ and
the entries of the matrices $K_{E_i}$ for $1\leq i \leq n$, with
rational coefficients, which is (the class of) an explicit
$\Sp(2n)$-invariant differential form in $\wt{A}(\X_\R)$.

In particular, if $k_i$ are nonnegative integers with $\sum k_i=
\dim{\X}=n^2+1$, the monomial $\x_1^{k_1}\cdots \x_n^{k_n}$ lies in
the ideal $I_n$. If $\x_1^{k_1}\cdots \x_n^{k_n}=\sum_{i=1}^n
e_i(\rX_n^2)f_i(\rX_n)$, then we have
\[
\wh{x}_1^{k_1}\wh{x}_2^{k_2}\cdots\wh{x}_n^{k_n}=
\sum_{i=1}^n(-1)^i\,\wt{c}_{2i}(\ov{\E},\ov{\E}^*)\wedge f_i(x_1,\ldots,x_n).
\]
Now if the top invariant form $\Om$ is defined as in \S \ref{hvb}, we
have shown that
\[
\wt{c}_{2i}(\ov{\E},\ov{\E}^*)\wedge f_i(x_1,\ldots,x_n)=r_i\,\Om
\]
for some rational number $r_i$. Therefore the arithmetic degree
\cite{GS1} of the above monomial satisfies
\[
\wh{\deg}(\wh{x}_1^{k_1}\wh{x}_2^{k_2}\cdots\wh{x}_n^{k_n})=
\frac{1}{2}\sum_{i=1}^n(-1)^i\,r_i\int_{\X(\C)}\Om=  
\frac{1}{2}\prod_{k=1}^n\frac{1}{(2k-1)!}
\sum_{i=1}^n(-1)^i\,r_i.
\]
We deduce the following analogue of \cite[Theorem 4]{T2}.
\begin{thm}
\label{mainthm} For any nonnegative integers  
$k_1,\ldots,k_n$ with $\sum k_i=n^2+1$, the arithmetic Chern number
$\dis \wh{\deg}(\wh{x}_1^{k_1}\wh{x}_2^{k_2}\cdots\wh{x}_n^{k_n}) $ is
a rational number.
\end{thm}

\subsection{Arithmetic Schubert calculus}
\label{asc}

For any partition $\l\in\G_n$ and $\om\in S_n$, define 
\[
\wh{\CS}_{\l,\om} = \CS_{\l,\om}(\wh{x}_1,\ldots,\wh{x}_n). 
\]
If $\l\in\G_n\ssm\D_n$, let $r_\l$ be the largest repeated part of
$\l$, and let $\ov{\l}$ be the partition obtained from $\l$ by
deleting two of the parts $r_{\l}$. For instance, if $\l=(8, 7,7,7,6,
3, 3, 2)$, then $\ov{\l}=(8,7,6,3,3,2)$.

If $\l\in \G_n\ssm\D_n$, then properties (b), (c) in \S \ref{definitions},
(\ref{modstr}), and (\ref{simple}) imply that 
\[
\wh{\CS}_{\l,\om} = \wh{\CS}_{\ov{\l},\om}
\wt{Q}_{r_\l,r_\l}(\wh{x}^2_1,\ldots,\wh{x}^2_n) 
= (-1)^{r_\l}\CS_{\ov{\l},\om}(x_1,\ldots,x_n)\wedge
\wt{c}_{2r_\l}(\ov{\E}, \ov{\E}^*).
\]
Since $\wh{\CS}_{\l,\om} \in a(\wt{A}(\X_\R))$ whenever $\l\in \G_n\ssm\D_n$,
we will denote these classes by $\wt{\CS}_{\l,\om}$.

The next theorem computes arbitrary arithmetic intersections in
$\wh{\CH}(\X)$ with respect to the splitting (\ref{bigiso}) induced 
by (\ref{splittingmap}), using the basis 
of symplectic Schubert polynomials.

\begin{thm}
\label{SPring}
Any element of the arithmetic Chow ring $\wh{CH}(\X)$ can be
expressed uniquely in the form $\dis \sum_{w\in
W_n}a_w\wh{\CS}_w+\eta$, where $a_w\in\Z$ and
$\eta\in\wt{A}(\X_{\R})$. For $u,v\in W_n$ we have
\begin{equation}
\label{maineq}
\wh{\CS}_u\cdot\wh{\CS}_v=\sum_{w\in W_n}c_{uv}^w\,\wh{\CS}_w+
\sum_{{\l\in \G_n\ssm\D_n}\atop{\om\in S_n}}c_{uv}^{\l\om}\,\wt{\CS}_{\l,\om},
\end{equation}
\[
\wh{\CS}_u\cdot \eta=\CS_u(x_1,\ldots,x_n)\wedge\eta,
\ \ \ \ \text{and}
\ \ \ \ \eta\cdot \eta'=(dd^c\eta)\wedge\eta',
\]
where $\eta$, $\eta'\in\wt{A}(\X_{\R})$ and the integers $c_{uv}^w$,
$c_{uv}^{\l\om}$ are as in {\em (\ref{structeq})}.
\end{thm}
\begin{proof} 
The first statement is a consequence of the splitting (\ref{bigiso}).
Equation (\ref{maineq}) follows from the formal identity
(\ref{structeq}) and our definitions of $\wh{\CS}_w$ and
$\wt{\CS}_{\l,\om}$. The remaining assertions are derived immediately from
the structure equations (\ref{modstr}).
\end{proof}

\subsection{The invariant arithmetic Chow ring}
\label{iacr}
The arithmetic Chow group $\wh{\CH}(\X)$ is not finitely generated, as
it contains the infinite dimensional real vector space $\wt{A}(\X_\R)$
as a subgroup. Following \cite[\S 6]{T2}, we can work equally well
with a finite dimensional variant of $\wh{\CH}(\X)$, obtained by
replacing the space $A(\X_\R)$ by a certain subspace of the
space of all $\Sp(2n)$-invariant differential forms on $\X(\C)$.

Recall the notation introduced in \S \ref{hvb} and \S \ref{acg}.  Let
$\Inv(\X_\R)$ denote the ring of $\Sp(2n)$-invariant forms in the
$\R$-subalgebra of $A(\X(\C))$ generated by the differential forms
$\ome_1\wedge\ov{\ome}_2$ for all $\ome_1,\ome_2$ in the set
$\{\ome_{ij},\,\ome^{pq}\ | \ i<j,\ p\leq q\}$.  
Define $\wt{\Inv}(\X_{\R})\subset \wt{A}(\X_{\R})$ to be the image of
$\Inv(\X_{\R})$ in $\wt{A}(\X_{\R})$.

\begin{defn}
The invariant arithmetic Chow ring $\wh{\CH}_{\inv}(\X)$ is the
subring of $\wh{\CH}(\X)$ generated by $\epsilon(\CH(\X))$ and
$a(\wt{\Inv}(\X_{\R}))$, where $\epsilon$ is the splitting map
(\ref{splittingmap}).
\end{defn}

There is an exact sequence of abelian groups
\[
0 \longrightarrow \wt{\Inv}(\X_{\R}) \stackrel{a}\longrightarrow 
\wh{\CH}_{\inv}(\X)
\stackrel{\zeta}\longrightarrow \CH(\X)\longrightarrow 0
\]
which splits under $\epsilon$, giving an
isomorphism of abelian groups
\[
\wh{\CH}_{\inv}(\X)\simeq \CH(\X)\oplus\wt{\Inv}(\X_{\R}).
\]
Theorem \ref{SPring} may be refined to an analogous statement
for the invariant arithmetic Chow ring $\wh{\CH}_{\inv}(\X)$, by
replacing the group $\wt{A}(\X_{\R})$ with $\wt{\Inv}(\X_{\R})$
throughout.

\subsection{Height computation}
\label{hgts}

The flag variety $\X$ has a natural pluri-Pl\"{u}cker embedding $j$ in 
projective space. The morphism $j$ may be defined as a composite
\[
\X\lhra F_{\text{SL}} \stackrel{\iota}{\lhra} \bP^N
\]
where $F_{\text{SL}}= \text{SL}_{2n}/P$ denotes the variety 
parametrizing all partial flags
\[
0=E_0\subset E_1\subset \cdots \subset E_n \subset E_{2n}=E
\]
with $\dim(E_i)=i$ for each $i$, and $\iota$ is a composition
of a product of Pl\"{u}cker embeddings followed by a Segre
embedding. Observe that $j$ is the embedding given by the line bundle
$Q=\bigotimes_{i=1}^n \det (E/E_i)$.  Let $\ov{\O}(1)$ denote the
canonical line bundle over the projective space $\bP^N$, equipped with
its canonical metric (so that $c_1(\ov{\O}(1))$ is the Fubini-Study
form).  The {\em Faltings height} of $\X$ relative to $\ov{\O}(1)$
(see \cite{GS1, Fa, BoGS}) is given by
\[
h_{\ov{\O}(1)}(\X)
=\wh{\deg}\left(\wh{c}_1(\ov{\O}(1))^{n^2+1}\vert \ \X \right).
\]
The pullback $j^*(\O(1))=Q$ is an isometry when $Q(\C)$ is equipped with
the canonical metric given by tensoring the induced metrics on the
determinants of the quotient bundles $E/E_i$ for $1\leq i \leq n$.
The short exact sequences
\[
\ov{\E}_i\ :\ 0 \to \ov{E}_{i-1}\to \ov{E}_i \to \ov{L}_{n+1-i}\to 0
\]
satisfy $\wt{c}_1(\ov{\E}_i)=0$, and hence $\wh{c}_1(\ov{E}_i)=
\wh{c}_1(\ov{E}_{i-1})-\wh{x}_{n+1-i}$.  It follows by induction that
$\wh{c}_1(\ov{E}_i)= - \wh{x}_{n+1-i} -\cdots - \wh{x}_n$, for $1\leq
i \leq n$.  We deduce that
\[
j^*(\wh{c}_1(\ov{\O}(1)))=\wh{c}_1(\ov{Q})=
-\sum_{i=1}^n\wh{c}_1(\ov{E}_i)=
\sum_{i=1}^n i\,\wh{x}_i
\]
and therefore
\begin{equation}
\label{hgteq}
h_{\ov{\O}(1)}(\Sp_{2n}/B)=
\wh{\deg}\left(\wh{c}_1(\ov{Q})^{n^2+1}\vert \ \X\right)=
\wh{\deg}\left(\left(\sum_{i=1}^n i\,\wh{x}_i\right)^{n^2+1}\right).
\end{equation}

We conclude from Theorem \ref{mainthm} and (\ref{hgteq}) that 
the height $h_{\ov{\O}(1)}(\Sp_{2n}/B)$ is a rational number. One may 
also derive this fact from the height formula in \cite{KK}.

\subsection{An example: $\Sp_{4}/B$}
\label{examp}

In this section we compute the arithmetic intersection numbers for the
classes $\wh{x}_i$ in $\wh{\CH}(\X)$ when $n=2$, so that $\X$ is the
variety parametrizing partial flags $0\subset E_1\subset E_2 \subset
E_4=E$ with $E_2$ Lagrangian.

Consider the differential forms
\[
\xi_1=c_1(\ov{E}_2^*) = \Om^{11}+2\,\Om^{12}+\Om^{22}
\]
and 
\[
\xi_2 = c_2(\ov{E}_2^*) = \Om^{11}\Om^{22}+ 2\,\Om^{11}\Om^{12} + 
2\,\Om^{12}\Om^{22}.
\]
Notice that $\xi_1^2=2\,\xi_2$. 
Over $\X$ we have the filtrations of hermitian vector bundles
\[
\ov{\E}_{\LG} : \ 0\subset \ov{E}_2 \subset \ov{E} \qquad \text{and}
\qquad \ov{\E} : \ 0\subset \ov{E}_1 \subset \ov{E}_2.
\]
Equation (\ref{ses}) gives 
\begin{align*}
\wt{c}(\ov{\E}_\LG)&=p_1(\ov{E}_2)+\H_3\,p_3(\ov{E}_2) = 
c_1(\ov{E}_2)+\H_3\left(c_1^3(\ov{E}_2)- 
3c_1(\ov{E}_2)c_2(\ov{E}_2)\right) \\
&= -\xi_1-\frac{11}{6}\xi_1^3+\frac{11}{2}\xi_1\xi_2 = 
-\xi_1+\frac{11}{6}\xi_1\xi_2
\end{align*}
and therefore
\[
\wt{c}(\ov{\E}_\LG)=-\Om^{11}-2\,\Om^{12}-\Om^{22}+11\,
\Om^{11}\Om^{12}\Om^{22}.
\]
On the other hand, Proposition \ref{compprop} gives $\wt{c}(\ov{\E})=
\wt{c}(\ov{\E}^*)= -\Om_{12}$. Using the Maurer-Cartan structure
equations for $\Sp_{2n}(\C)$, we find that
%
%
\[
d\ome_{12} = \ov{\partial}\ome_{12}
=\frac{1}{\gamma}(\ome^{11}\wedge\ov{\ome}^{12} + 
\ome^{12}\wedge\ov{\ome}^{22})
\]
\[
d\ov{\ome}_{12} = \partial\ov{\ome}_{12}
=-\frac{1}{\gamma}(\ome^{22}\wedge\ov{\ome}^{12} + 
\ome^{12}\wedge\ov{\ome}^{11})
\]
and hence
\[
dd^c(\Om_{12})= \gamma^2\partial\ov{\partial}
(\ome_{12}\wedge\ov{\ome}_{12}) = \gamma^2(\ov{\partial}\ome_{12}
\wedge\partial\ov{\ome}_{12})=
\Om^{12}(\Om^{11}+\Om^{22}).
\]
We deduce from the above calculations and (\ref{keypart}) that
\begin{gather*}
\wt{c}(\ov{\E}, \ov{\E}^*) = -\xi_1 -2\,\Om_{12} - 2\,\Om_{12}\,\xi_2
+11\, \Om^{11}\Om^{12}\Om^{22}
+\Om_{12}\Om^{12}(\Om^{11}+\Om^{22}) \\
= -\xi_1 -2\,\Om_{12} - 2\,\Om_{12}\Om^{11}\Om^{22}
-3\, \Om_{12}\Om^{11}\Om^{12} - 3\, \Om_{12}\Om^{12}\Om^{22}
+11\, \Om^{11}\Om^{12}\Om^{22}.
\end{gather*}

Let $a$ and $b$ be nonnegative integers with $a+b=5$.  The Bott-Chern
form $\wt{c}(\ov{\E}, \ov{\E}^*)$ is the key to computing any
arithmetic intersection $\wh{x}_1^a\wh{x}_2^b$, following the
algorithm of \S \ref{compaa}. The result will be a multiple of the
class of $\Om = \Om_{12}\Om^{11}\Om^{12}\Om^{22}$ in the arithmetic
Chow group $\wh{\CH}^5(\X)$. For instance, we compute that
\[
\wh{x}_1^3\wh{x}_2^2 = \wh{x}_1\cdot e_2(\wh{x}_1^2,\wh{x}_2^2) = 
(-\Om_{12}+\Om^{11} + \Om^{12})\wedge \wt{c}_4(\ov{\E}, \ov{\E}^*) =
-16\,\Omega,
\]
while
\[
\wh{x}_1^2\wh{x}_2^3 = \wh{x}_2\cdot e_2(\wh{x}_1^2,\wh{x}_2^2)
=(\Om_{12}+\Om^{12} + \Om^{22})
\wedge \wt{c}_4(\ov{\E}, \ov{\E}^*) = 6\,\Omega.
\]
We similarly find 
\[
\wh{x}_1^5 = 10\,\Om, \ \ \quad
\wh{x}_1^4\wh{x}_2 = -8\,\Om, \ \ \quad
\wh{x}_1\wh{x}_2^4 = 26\,\Om, \ \ \quad
\wh{x}_2^5 = 0.
\]
For the Faltings height of $\X$, we conclude that
\[
h_{\ov{\O}(1)}(\Sp_{4}/B)=\wh{\deg}\left((\wh{x}_1+2\wh{x}_2)^5\right)
= \wh{\deg}(1850\,\Omega) = 925\,\int_{\X(\C)}\Om =\frac{925}{6}.
\]
Kaiser and K\"ohler have proved a cohomological formula for 
the height of generalized flag varieties \cite[Thm.\ 8.1]{KK}.
One can check that in the case of  $\Sp_{4}/B$, their result 
agrees with the above computation.

%
%

\end{document}